\documentclass{amsart}

\usepackage{todonotes}
\usepackage{amssymb,latexsym,amsthm,amsmath}
\usepackage{tikz,tikz-cd}
\usepackage{tikz-cd}

\newtheorem{theorem}{Theorem}[section]
\newtheorem{prop}[theorem]{Proposition}
\newtheorem{lemma}[theorem]{Lemma}
\newtheorem{cor}[theorem]{Corollary}

\newtheorem{defn}[theorem]{Definition}

\newtheorem{rmk}[theorem]{Remark}

\newtheorem{eg}[theorem]{Example}

\newcommand{\G}{{\Gamma}}

\newcommand{\natls}{{\mathbb N}}

\newcommand\FF{{\mathcal F}}

\newcommand\LL{{L}}
\newcommand\MM{{\mathcal M}}

\newcommand\PP{{\mathcal P}}

\newcommand\RR{{\mathcal R}}

\newcommand\PMF{{\PP\kern-2pt\MM\FF}}

\newcommand\PML{{\PP\kern-2pt\MM\LL}}

\newcommand\Z{{\mathbb Z}}
\newcommand\R{{\mathbb R}}

\renewcommand{\L}{{\Lambda}}

\newcommand\til{\widetilde}

\numberwithin{equation}{section}
\allowdisplaybreaks
\usepackage{amsfonts,amssymb,verbatim,amsmath,amsthm,latexsym,textcomp,amscd}
\usepackage{latexsym,amsfonts,amssymb,epsfig,verbatim}
\usepackage{amsmath,amsthm,amssymb,latexsym,graphics,textcomp, hyperref}
\usepackage{graphicx}
\usepackage{color}
\usepackage{url}
\usepackage{psfrag}

\usepackage{cite}
\usepackage{xcolor}
\usepackage{tikz-cd}
\usepackage{float}
\usepackage{bm}
\usepackage{mathtools}
\usepackage{graphicx}
\usepackage{enumerate}
\usepackage{enumitem}

\usepackage{amsfonts,amssymb,verbatim,amsmath,amsthm,latexsym,textcomp,amscd,hyperref}
\usepackage{latexsym,amsfonts,amssymb,epsfig,verbatim}
\usepackage{graphicx}
\usepackage{color}
\usepackage{url}
\usepackage{pgfplots}
\usepackage{tikz}
\usepackage{pdfpages}
\usepackage{subcaption}

\graphicspath{ {./images/} }

\newtheorem*{theorem*}{Theorem}
\newtheorem*{lemma*}{Lemma}
\newtheorem*{proposition*}{Proposition}
\newtheorem*{corollary*}{Corollary}
\newtheorem{theoremA}{Theorem}

\theoremstyle{definition}
\newtheorem{remark}[theorem]{Remark}

\newtheorem{question}[theorem]{Question}

\def\RR{\mathbb R_{\geq 0}}
\def\dl{\delta}

\def\pr{^\prime}

\def\ri{\rightarrow}

\def\sse{\subseteq}
\def\gm{\gamma}
\def\bt{\beta}
\def\al{\alpha}

\def\pa{\partial}
\def\map{\rightarrow}

\def\L{\mathcal{L}}

\def\Y{\mathcal{Y}}
\def\G{\mathcal{G}}

\def\R{\mathbb {R}}

\def\N{\mathbb {N}}

\def\Z{\mathbb {Z}}

\title{Finiteness of Cannon--Thurston fibers}

\subjclass[2010]{20F65, 20F67 (Primary), 57M50 (Secondary)}
\keywords{Hyperbolic metric space, hyperbolic group, Cannon--Thurston map, Metric graph bundle}

\author{Indranil Bhattacharyya}
\address{Tata Institute of Fundamental Research (TIFR), Mumbai, India}
\email{indranil38e2@gmail.com}

\author{Rakesh Halder} 
\address{Tata Institute of Fundamental Research (TIFR), Mumbai, India}
\email{rhalder.math@gmail.com}

\author{Nir Lazarovich}

\address{Department of Mathematics
	Technion
	Haifa 32000
	Israel}

\email{lazarovich@technion.ac.il}

\author{Mahan Mj}

\address{School of Mathematics, Tata Institute of Fundamental Research, Mumbai-40005, India}
\email{mahan@math.tifr.res.in}
\email{mahan.mj@gmail.com}
\urladdr{http://www.math.tifr.res.in/~mahan}
\date{\today}

\thanks{IB, RH, MM were supported by the Department of Atomic Energy,
Government of India, under project no.12-R\& D-TFR-5.01-0500 as also by an endowment of the
Infosys Foundation. NL was partially supported by the Israeli Science Foundation (grant no. 1576/23). 
Most of this work was done during a visit of NL to Tata Institute of Fundamental Research in 2026. Part of this work is also supported  by the National Science
Foundation under Grant No. DMS-1928930: MM was in residence at
the Simons Laufer Mathematical Sciences Institute in Berkeley, California, during
the Spring 2026 semester. }

\begin{document}
\begin{abstract}
Let $Y\to X$ be a proper map between proper 
hyperbolic metric spaces. 
A Cannon--Thurston map is a continuous extension $\partial Y \to \partial X$.
We prove that in most known settings in which a Cannon--Thurston map exists it is uniformly finite-to-one.
This answers a question of Swarup and generalizes previous results of Cannon--Thurston, Kapovich--Lustig,
Dowdall--Kapovich--Taylor and Ghosh.
\end{abstract}

\maketitle

\section{Introduction}
Let $M$ be a closed hyperbolic 3-manifold fibering over the circle with fiber $F$. Cannon and Thurston \cite{CTpub} proved that the inclusion map $i: \til F \to \til M$ extends continuously to $\partial i: \partial\til F \to \partial\til M$. Further, they identified point-preimages of $\partial i$ precisely in terms of stable and unstable laminations on $F$. In particular, they showed that $\partial i$ is uniformly finite-to-one. 

More generally, let $Y\to X$ be a proper map between proper hyperbolic spaces. A continuous extension $\partial Y\to \partial X$, if exists, is called a \emph{Cannon--Thurston map}. In \cite[Question 1.20]{bestvinaprob}, Swarup asked the following:

\begin{question}\label{qsn-main}
Let
	 $\Pi: X \to T$ be a tree of hyperbolic metric spaces such that the inclusion of edge spaces into vertex spaces are quasi-isometric (qi) embeddings 
	as in Bestvina--Feighn's combination theorem \cite{BF}.
	Further, assume that  $X$ is hyperbolic. Let $v$ be
	a vertex of $T$ and $(X_v , d_v)$ denote the hyperbolic vertex space
	corresponding to $v$. It was proven in \cite{mitra-trees} that $i : X_v \to X$ extends continuously to
$\partial i : \partial X_v \to \partial X$. 
 Is $\partial i_v : \partial X_v \to \partial X$ finite-to-one?
\end{question}	
The main theorem of this paper answers Question~\ref{qsn-main} affirmatively. In fact we prove the following more general statement
(see Theorems~\ref{thm-finite to one TMS ray}, \ref{thm-finite to one MB gen} and Corollaries~\ref{thm-gog vertex}, ~\ref{cor-normal}).
\begin{theoremA}\label{thm-main}
Let $X,Y$ be bounded valence graphs, satisfying one of the following:
\begin{enumerate}
    \item $X$ is the Cayley graph of a hyperbolic group; $Y$ is the Cayley graph of a normal hyperbolic subgroup. 
    \item $X$ is a hyperbolic tree of hyperbolic spaces such that the vertex spaces are uniformly hyperbolic and the inclusions of edge spaces into vertex spaces are uniformly qi embeddings
(as in Bestvina--Feighn's combination theorem \cite{BF}); $Y$ is a vertex space.
    \item $X$ is a hyperbolic metric graph bundle of uniformly hyperbolic spaces with uniformly coarse surjective barycenter maps as in Mj--Sardar \cite{pranab-mahan}; $Y$ is a fiber space. 
\end{enumerate}
Then the Cannon--Thurston map $\partial Y\to \partial X$ of the natural inclusion $Y\to X$ is uniformly finite-to-one.
\end{theoremA}

The existence of a Cannon--Thurston map in the cases of the theorem was proven respectively in \cite[Theorem 4.3]{mitra-ct}, \cite[Theorem 3.10]{mitra-trees} and \cite[Theorem 5.3]{pranab-mahan}.
Case (1) follows from Case (3), and generalizes previous results of Kapovich and Lustig \cite[Theorem A]{kl15}, Dowdall, Kapovich and Taylor \cite[Theorem 6.3]{dkt} and Ghosh  \cite[Theorem $4.8$]{pritam-ctfinite}.
Case (2) answers a question of Swarup \cite[Question 1.20]{bestvinaprob}.

Unlike previous approaches, we prove uniform finiteness of fibers of the Cannon--Thurston map directly without going through the description of Cannon--Thurston laminations as in \cite{kl15,dkt,mj-rafi,pritam-ctfinite}. We sketch the main idea in the case of a hyperbolic metric bundle over the ray:

\bigskip

\noindent {\bf Sketch of the main idea:}
Let $T=[0, \infty)$ be a ray with vertices at integer points $\{0\}\cup\natls$. Let $X \to T$ be a tree of (uniformly) $\delta-$hyperbolic spaces such that edge to vertex space maps are uniform quasi-isometries (and not just qi embeddings). Hence there exist uniform 
quasi-isometries $\phi_n: X_n \to X_{n+1}$. Let $\partial \phi_n: 
\partial  X_n \to \partial  X_{n+1}$ denote the induced boundary homomorphisms.
Let $\partial \Phi_n: 
\partial  X_0 \to \partial  X_{n}$ be given by $\partial \Phi_n = 
\partial \phi_{n-1} \circ \cdots\circ\partial \phi_0$.
Suppose $X$ is hyperbolic.

Let $i: X_0 \to X$ denote the inclusion of the base vertex space. Let
$\partial i: \partial X_0 \to \partial X$ denote the Cannon--Thurston map \cite{mitra-trees}.  
Let $\zeta\in \partial X$. Our goal is to uniformly bound the size  $|\partial i^{-1}(\zeta)|.$ We will do so by uniformly bounding the size of any finite subset $A\subseteq \partial i^{-1}(\zeta)$.

Let $\xi_1, \xi_2, \xi_3 \in A$ be distinct points.
 The triple of points $\partial \Phi_n(\xi_1), \partial \Phi_n(\xi_2),\partial \Phi_n(\xi_3)$ defines a coarse barycenter 
   $Bary_{X_n}(\partial \Phi_n(\xi_1), \partial \Phi_n(\xi_2),\partial \Phi_n(\xi_3))$ (Definition~\ref{barycenter map}) in each fiber $X_n$. 
  The ray $$r=r(\xi_1,\xi_2,\xi_3):\{0\}\cup \natls \to X\quad  \text{ defined by} \quad r(n) = Bary_{X_n}(\partial \Phi_n(\xi_1), \partial \Phi_n(\xi_2),\partial \Phi_n(\xi_3))$$ is a quasigeodesic ray (with uniform quasi-geodesic constants).
  Moreover, since  $\xi_1,\xi_2,\xi_3\in \partial i^{-1}(\zeta)$ we have $r(\infty) = \zeta$. 
  
  By the previous paragraph, it follows that for any two triples (of distinct points) $\xi_1,\xi_2,\xi_3$ and $\xi'_1,\xi'_2,\xi'_3$ in $A$ the rays $r=r(\xi_1, \xi_2, \xi_3)$ and $r'=r(\xi'_1, \xi'_2. \xi'_3)$ are asymptotic. 
  That is, there exists $D\ge 0$ (independent of $A$) and $N=N(\xi_1,\xi_2,\xi_3,\xi'_1,\xi'_2,\xi'_3)$ such that for all $n\ge N$, $d_{X_n} (r(n),r'(n))\le D$.
  Since $A$ is finite there exists $M\in \natls$ such that the distance between barycenters of distinct triples in 
  $\partial \Phi_M(A)$ is bounded above by $D$ in
  $\partial X_{M}$. 
  By compactness of $\partial X_M$ and bounded valence of $X_M$, it then follows that $A$ is of   bounded size (see Proposition~\ref{prop-bdd nm barycenters in bdd ball} for further detail). 
 
 To deal with the case of trees of hyperbolic metric spaces, we shall need the recent technology of boundary flows (see Section~\ref{sec-bflow}) in the context
 of trees of hyperbolic spaces developed by Kapovich and Sardar in \cite{ps-kap}.
 They generalized the more straightforward case above where the technology was already developed 
 in \cite{pranab-mahan}.

As an application we give a new and short proof of the main technical theorem of \cite{NirMaMj-commen}
(see Theorem~\ref{thm-lmm}). We also recover a theorem of Thurston \cite{thurston-slither1,calegari-dunfield} that proves that a Cannon-Thurston map exists and is finite-to-one for leaves of foliated hyperbolic 3-manifolds slithering over the circle, see Section~\ref{sec-newproofs}.\\

\noindent {\bf A historical note:} Swarup's question \cite[Question 1.20]{bestvinaprob} had two parts. The first part asked for a description of point-preimages, i.e. Cannon--Thurston laminations in the terminology of \cite{mj-rafi}. The second part is addressed by Theorem~\ref{thm-main}. The question was stated 
originally in
\cite{bestvinaprob} in the somewhat more restrictive group-theoretic setting of a finite graph of hyperbolic groups satisfying the qi-embedded condition.
The recent exhaustive monograph \cite{ps-kap} (especially Chapter 8) does furnish a description, building on related work in \cite{mitra-endlam,pranab-mahan}. However, the description in 
\cite{mitra-endlam} and \cite{ps-kap} is not enough to address Question~\ref{qsn-main}. The description of Cannon--Thurston laminations in \cite{mitra-endlam} was leveraged in the special case where one has a finitely generated free normal subgroup of a hyperbolic group \cite{kl15,dkt,pritam-ctfinite} to answer positively the analog of 
Question~\ref{qsn-main}. But this approach makes essential use of the fairly sophisticated index theory of free group automorphisms developed in \cite{chl07,ch14}. All these techniques involving laminations and index theory were inspired by Thurston's theory of pseudo-Anosov surface diffeomorphisms
and geodesic laminations.
The key new contribution of this paper is that it largely circumvents the use of any lamination-related machinery and addresses Question~\ref{qsn-main} directly in a fairly general context.\\

\noindent {\bf Organization of the paper:} In Section~\ref{sec-prel} we discuss some preliminary material on hyperbolic spaces and Cannon--Thurston maps. We single out Proposition~\ref{prop-bdd nm barycenters in bdd ball} from this section as a basic fact that is used crucially in the rest of the paper. Section~\ref{sec-treebdls} is devoted to trees of hyperbolic spaces. The main theorem of this paper, Theorem~\ref{thm-finite to one TMS ray} is proven there, answering Question~\ref{qsn-main}. Section~\ref{sec-mgb} proves the analogous result for metric graph bundles. Section~\ref{sec-further extension} provides further extensions and generalizations. 
In particular, we give a quick short proof of one of the main theorems of \cite{NirMaMj-commen}.

\section{Preliminaries}\label{sec-prel}

\subsection{Hyperbolic metric spaces and barycenter maps}

We refer the reader to \cite{gromov-hypgps},\cite[Chapter III.H]{bridson-haefliger}, \cite{abc} for basics on  hyperbolic spaces and groups and their Gromov boundaries. We assume henceforth that all spaces apart from underlying Bass-Serre trees  are   proper.
Also, for brevity of notation, we shall use $k-$qi embeddings to mean $(k,k)-$qi embeddings and $k-$quasi-geodesic to mean $(k,k)-$quasi-geodesic.

The following lemma collects together \cite[Lemmas 1.17, 3.2, 3.3, III.H]{bridson-haefliger}

\begin{lemma}\label{lem-barycenter}
Let $X$ be a proper $\dl$-hyperbolic  metric space. Let $ \pa X$ denote its boundary. 
Given  $\xi_1 \neq \xi_2\in \partial X$,  there exists a bi-infinite geodesic  $(\xi_1,\xi_2)$ in $X$ joining $\xi_1$ and $\xi_2$.

There exist $D_{\ref{lem-barycenter}}(\dl)\ge0$ and $R_{\ref{lem-barycenter}}(\dl)\ge0$ depending only on $\dl$ such that the following holds.
For any triple $\xi_1 , \xi_2 , \xi_3$ of distinct points in $\pa X$ there exists $x\in X$ such that the $D_{\ref{lem-barycenter}}(\dl)$-neighborhood of $x$ intersects  $(\xi_i, \xi_j)$  for $i\ne j$ and $i,j\in\{1,2,3\}$. Moreover, $x$ is coarsely unique, i.e.\ for any  $x'\in X$ satisfying the above condition,  $d(x,x')\le R_{\ref{lem-barycenter}}(\dl)$. We refer to $x$ as a 
 \emph{barycenter} of the triple $\Xi= (\xi_1,\xi_2,\xi_3)\in\pa^3X$.
\end{lemma}

Let $X$ be a proper $\dl$-hyperbolic geodesic metric space. Let $$(\xi_1,\xi_2,\xi_3)\in\pa^3 X:=\{(\eta_1,\eta_2,\eta_3)\in\pa X\times\pa X\times\pa X:\eta_1\ne\eta_2\ne\eta_3\}.$$ An \emph{ideal triangle} with vertices $\xi_1,\xi_2,\xi_3$, denoted by $\triangle(\xi_1,\xi_2,\xi_3)$, is a union of
bi-infinite geodesics $(\xi_i, \xi_j)$ for $i\ne j$ and $i,j\in\{1,2,3\}$. 
The following definition is motivated by Lemma \ref{lem-barycenter} (see \cite[p. $1668$]{pranab-mahan}).
\begin{defn}[Barycenter map]\label{barycenter map}
Let $X$ be a proper $\dl$-hyperbolic geodesic metric space. For any $\Xi= (\xi_1,\xi_2,\xi_3)\in\pa^3X$, let  $x$ be  a barycenter as in Lemma \ref{lem-barycenter}. We thus have a coarsely well-defined \emph{barycenter map} $Bary_X:\pa^3X\map X$  sending $\Xi$ to its barycenter. 
	
	For $D \ge 0$, a point $z \in X$ is called a $D$-barycenter of the ideal triangle $\triangle(\xi_1, \xi_2, \xi_3)$ if $z$ is $D$-close to each bi-infinite geodesic joining $\xi_i,\xi_j$ for distinct $i, j \in \{1, 2, 3\}$. Further,  $z \in X$ is called a $D$-barycenter of the triple of points $\xi_1 \neq  \xi_2 \neq  \xi_3
    \in \partial X$, if it is a $D$-barycenter of the
    ideal triangle $\triangle(\xi_1, \xi_2, \xi_3)$.
\end{defn}

 It is  a fact that any two $D$-barycenters of an ideal triangle are $D'$-close for some $D'\ge0$ depending on $D$ and $\dl$ (see \cite[Lemma 2.7]{pranab-mahan} for instance). Hence as a consequence of  stability of quasigeodesics we have the following
(the proof is the same as that of \cite[Lemma 3.5]{mitra-trees}).

\begin{lemma}\label{lem-barycenter close}
Let $\dl\ge0,L\ge1$. There exists $R\ge0$ such that the following hold.
Let $X,Y$ be  $\dl$-hyperbolic metric spaces. 
	Let $\phi:X\to Y$ be an $L$-qi embedding. Let $\pa \phi:\pa X\to\pa Y$  denote the induced topological embedding between  their boundaries \cite[Theorem $3.9$]{bridson-haefliger}.
    Let $(\xi_1,\xi_2,\xi_3)\in\pa^3 X$ and $x$ be a barycenter of $(\xi_1,\xi_2,\xi_3)$ and $y$ be a barycenter of $(\pa \phi(\xi_1),\pa \phi(\xi_2),\pa \phi(\xi_3))$. 
	Then $d_Y(\phi(x),y)\le R$.
\end{lemma}

We conclude this subsection with the following result, which gives a uniform bound on the number of boundary points such that the barycenters of all distinct triples remain within a bounded set.

\begin{prop}\label{prop-bdd nm barycenters in bdd ball}
	Given $\dl\ge0$, $D\ge0$ and $R\ge0$ there is a constant $D'$ such that the following holds. Let $X$ be a $\dl$-hyperbolic  graph such that  any vertex has valence bounded by $D$. Let $u \in X$ be any vertex.
	Let $A \subset \partial X$ be a set such that for any distinct $\xi_1,\xi_2
, \xi_3
\in A$, $Bary_X(\xi_1, \xi_2,
 \xi_3)\cap B(u;R)\ne\emptyset$.
	Then the cardinality of $A$ is bounded by $D'$.
\end{prop}	

\begin{proof} Since  $Bary_X(\xi_1, \xi_2,
	\xi_3)\cap B(u;R)\ne\emptyset$ for any distinct $\xi_1 , \xi_2, \xi_3
	\in A$, it follows that the Gromov inner product
	$\langle \xi_i, \xi_j \rangle_u \leq R'$ for some $R' = R'(R, \delta)$ for all $i\ne j\in \{1,2,3\}$.
	Let $\alpha_i =[u, \xi_i)$ denote geodesic rays from $u$ to $\xi_i$.
	Let $S(u, R'+10\delta)$ denote the sphere of radius $R'+10\delta$ about $u$. 
	Since $\langle \xi_i, \xi_j \rangle_u \leq R'$, 
	$\alpha_i \cap S(u, R'+10\delta) \neq \alpha_j \cap S(u, R'+10\delta)$ for
	$i \neq j$. Hence the cardinality of $A$ is bounded by the cardinality of
	$S(u, R'+10\delta)$.
	Finally,
	since  any vertex in $X$ has valence bounded by $D$, 
	the cardinality of
	$S(u, R'+10\delta)$ is bounded by $D^{R'+10\delta}$.
	Choosing $D'=D^{R'+10\delta}$ completes the proof.
\end{proof}

\subsection{Cannon--Thurston maps}

\begin{defn}\label{CT-map}\cite{mitra-ct,mitra-trees}
	Let   $i: (Y,d_Y) \to (X,d_X)$ be a proper injective map of hyperbolic metric spaces.
	We say that  $i:Y\map X$ admits a \emph{Cannon--Thurston map} if $i$ extends continuously to a map $\partial i : \partial Y \to \partial X$.
	
	Suppose $H<G$ are hyperbolic groups. Let $\partial H, \partial G$ denote their Gromov boundaries.
	Choose a finite generating set for $G$ containing a 
	finite	generating set for $H$, so that we have an inclusion $i: \Gamma_H \to \Gamma_G$ of Cayley graphs  $\Gamma_H, \Gamma_G$ with respect to these generating sets. We say that the inclusion $H\map G$ admits a \emph{Cannon--Thurston  map} if $i: \Gamma_H \to \Gamma_G$ admits a Cannon--Thurston map.
\end{defn}

\noindent{\bf Notation.} When we want to emphasize the spaces in question, we denote the inclusion of $Y$ into $X$ by $i_{Y,X}:Y\to X$. We shall denote the corresponding Cannon--Thurston map, if it exists,  by $\pa i_{Y,X}:\pa Y\to\pa X$.

The following lemma asserts the basic fact that a composition of Cannon--Thurston maps is a Cannon--Thurston map.
\begin{lemma}\label{lem-functoriality of CT map}
Let $Z\sse Y\sse X$ be inclusions of hyperbolic graphs. (Here,  all spaces are equipped with their natural graph metrics.) Assume that the inclusions $i_{Z,Y}:Z\map Y$ and $i_{Y,X}:Y\map X$ admit Cannon--Thurston maps. Then the inclusion $i_{Z,X}:Z\map X$ admits a Cannon--Thurston map given by $\pa i_{Z,X}=\pa i_{Y,X}\circ\pa i_{Z,Y}$.\quad \qedsymbol
\end{lemma}

\section{Trees of spaces}\label{sec-treebdls}
\subsection{Trees of hyperbolic metric spaces}\label{sec-tree}
We recall the notion of trees of hyperbolic metric spaces
 \cite{BF} as adapted in \cite{mitra-trees}. 

\begin{defn}\label{defn-tree-of-sps}
Let $T$ be a simplicial tree and $(X, d)$  a geodesic metric space. Let $\Pi:X\to T$ be a surjective,
	$1$-Lipschitz map. 
	 For any vertex $b\in T$, let $X_b=\Pi^{-1}(b)$ and let $d_b$ denote the length metric
	on $X_b$ induced from $X$. Also for any unoriented edge $e$ of $T$ let $m_e$ denote the midpoint of $e$.
	Let $X_e$ denote $X_{m_e}$ and let $d_e$ denote the metric $d_{m_e}$.
	
We say that $\Pi:X\map T$ is a \emph{tree of hyperbolic metric spaces satisfying the q(uasi) i(sometrically) embedded condition} if there exist $\delta_0 \geq 0$,  $L_0\geq 1$ 
and a proper function $f: \RR \ri\RR$ (i.e. $f(r)\to\infty$ as $r\to\infty$) such that the following hold:
	\begin{enumerate}
		\item  For all vertices $b\in T$ and edges $e$ of $T$, $(X_b, d_b)$
		and $(X_e, d_e)$ are $\delta_0$-hyperbolic.
		\item The inclusion maps of $(X_b, d_b)$
		and $(X_e, d_e)$ in $(X,d)$ are uniformly proper as measured by $f$,
		i.e.\ for $b$ a vertex of $T$, $u,w  \in X_b$ and $C \in \RR$, $d_X(u,w) \leq C$ implies $d_b(u,w) \leq f(C)$; similarly for $e$ an edge of $T$ and 
	$u,w  \in X_e$.
		\item For any (unoriented) edge $e$ of $T$ joining two vertices $b, b'$, there is a map
		$f_e: X_e \times[0,1]\ri\Pi^{-1}(e)\sse X$ such that the following hold.
		\begin{enumerate}
			\item $\Pi\circ f_e$ is  the projection $X_e\times [0,1]\to [0,1]$ composed with the identification $[0,1]\to e$.

			\item $f_e$ restricted to $X_e\times(0,1)$ is an isometry onto the pre-image (under $\Pi$) of the interior of $e$ equipped with the
			path metric.
			
			\item $f_e$ restricted to $X_e \times \{0\}$ and $X_e \times \{1\}$ are $L_0$-qi embeddings into $X_v$ and $X_w$ respectively. We denote these restriction maps by $f_{e,v}$ and $f_{e,w}$ respectively.
				\end{enumerate}
	\end{enumerate}
	
The constants $\dl_0$, $L_0$ and $f$ are called parameters of the tree of spaces.
\end{defn}

For the rest of this section, we assume that $\Pi:X\to T$ is a tree of hyperbolic spaces satisfying the qi embedded condition with  parameters $\dl_0,L_0,f$.
We further assume that $X$ is $\dl_0$-hyperbolic.

The first part of the following theorem was behind Swarup's question~\ref{qsn-main}.

\begin{theorem}\label{thm-existence CT TMS}\cite{mitra-trees}\cite[Theorem 8.13]{ps-kap}
Let $\Pi:X\to T$ be a tree of hyperbolic spaces satisfying the qi embedded condition. Assume that $X$ is hyperbolic. Then 
for any vertex $b$ in $T$, the inclusion $i: X_b\to X$ admits a Cannon--Thurston map. 

More generally, for
any subtree $S$ of $T$, $\Pi^{-1}(S)$  is hyperbolic (with hyperbolicity constant depending only on that of $X$)  and the inclusion $\Pi^{-1}(S)\to X$ admits a Cannon--Thurston map.
\end{theorem}

\begin{defn}[QI sections]\label{defn-lift} Let
 $T\pr$ be a subtree of $T$ and $k\geq 1$.
A $k$-{\em qi section} of $T\pr$ is a  $k$-qi embedding $s:T\pr\ri X$ such that
$\Pi\circ s$ is the identity map on the vertex set $V(T\pr)$ of $T\pr$.
\end{defn}

\begin{lemma}\label{lem-lifts imp close in fiber} Let $\Pi: X \to T$ be a tree of hyperbolic metric spaces satisfying the qi-embedded condition such that $X$ is hyperbolic.
 Given $k\ge1$, there is $R_{\ref{lem-lifts imp close in fiber}}=R_{\ref{lem-lifts imp close in fiber}}(k,\dl_0)\ge0$ satisfying the following. \\
 Let $T\pr \subset T$ be any geodesic ray in $T$ with vertices at $n\in\{0\}\cup\N$.
 Denote the vertex space over $n\in\{0\}\cup\N$ by $X_n$.
Let $\al$ and $\bt$ be two $k$-qi sections of $T\pr$ converging to the same point in $\pa X$. Then there exists $N \in \natls$ such that for  $n\geq N$, $d_{X_n}(\al(n),\bt(n))\le R_{\ref{lem-lifts imp close in fiber}}$.
\end{lemma}

\begin{proof} Since $\al$ and $\bt$ are $k$-quasigeodesics converging to the same point in $\pa X$, there exists 
	$D\ge0$ depending on $k$ and $\dl_0$ and $N \in \natls$ satisfying the following: for  $n\geq N$,
 there exists $m\in\N$ such that $d_X(\al(n),\bt(m))\le D$. Since $\Pi$ is 1-Lipschitz, $|n-m|=d_T(\Pi(\al(n)),\Pi(\bt(m)))\le D$. By the triangle inequality, $d_X(\al(n),\bt(n))\le d_X(\al(n),\bt(m))+d_X(\bt(m),\bt(n))\le kD+k+D$. Since vertex spaces are $f$-proper embedding in $X$, $d_{X_n}(\al(n),\bt(n))\le f(kD+k+D)$. Setting $f(kD+k+D)=:R_{\ref{lem-lifts imp close in fiber}}(k)$ completes the proof.
\end{proof}

\subsection{Boundary flows}\label{sec-bflow}
Following \cite[Definition 4.3]{ps-conical} and \cite[Section 3.3.4]{ps-kap}, we now define the {\em boundary flow} of points in the boundary of vertex spaces as follows.
\begin{defn}\label{defn-bflow} Let $\Pi: X \to T$ be a tree of hyperbolic metric spaces.
	Let $e=[u,v]$ be an edge of $T$. Let $f_{e,u}, f_{e,v}$ be the qi-embeddings
 of $X_e$ into $X_u, X_v$ respectively. Let 
 $\partial f_{e,u}:\partial X_{e}\to \pa X_u$ and $\partial f_{e,v}:\partial X_{e}\to \pa X_v$ denote the induced embeddings of the boundaries. 
 If $\xi_u\in \partial X_u$ is in
	the image of $\partial f_{e,u}$, then
	$\xi_v:=\partial f_{e,v} \circ (\partial f_{e,u})^{-1}(\xi_u)$ is called the boundary
	flow of $\xi_u$ to  $\partial X_v$. We let $\pa\Phi_{uv}$ denote the resulting
	\emph{partially defined map} from (possibly a subset of) $ \pa X_u$ to
	$ \pa X_v$ so that $\pa\Phi_{uv}(\xi_u) = \xi_v$.
	
Let $u,v\in V(T)$ be any two vertices. Let  $u=u_0,u_1, \cdots, u_n=v$ be the
	consecutive vertices on the geodesic between $u, v$ in $T$. Let $\xi_0\in \partial X_u$. We say that
	$\xi_n\in \partial X_v$ is the boundary flow of $\xi_0$ to $\partial X_v$ if there exist (necessarily unique)
 $\xi_i\in \partial X_{u_i}$, $1\leq i\leq n-1$ such that $\xi_i$ is the boundary flow
	of $\xi_{i-1}$ to $\partial X_{u_i}$ for all $1\leq i\leq n$.
	We let $\pa\Phi_{uv}$ denote the resulting
	\emph{partially defined map} from (possibly a subset of) $ \pa X_u$ to
	$ \pa X_v$ so that $\pa\Phi_{uv}(\xi_u) = \xi_v$.
\end{defn}

Let $e=[u,v]$ be an edge of $T$ as above. Let $\al_u$ be a bi-infinite geodesic  in $X_u$ such that  $\al_u(\pm\infty)$ admit a boundary flow to $\partial X_v$. Let $\al_e$, $\al_v$ be bi-infinite geodesics in $X_e$ and $X_v$ respectively such that $\pa f_{e,u}(\al_e(\pm\infty))=\al_u(\pm\infty)$ and $\pa f_{e,v}(\al_e(\pm\infty))=\al_v(\pm\infty)$. 

\begin{defn}[Bi-infinite Ladder]\label{def-ladder}
Let $u\in V(T)$. Suppose $\al=\alpha_u$ is a bi-infinite geodesic  in $X_u$ joining $\xi$ and $\xi'$ in $\pa X_u$. Let $[u,\eta)\sse T$ be a geodesic ray in $T$ where $\eta\in\pa T$. Suppose further that  $\pa\Phi_{uv}(\xi)=\xi_v$, $\pa\Phi_{uv}(\xi^{'})=\xi'_v$ are defined  for all vertices $v$ in $[u,\eta)$. Let $\alpha_v$ be a bi-infinite geodesic joining $\xi_v$ and $\xi'_v$  in $ X_v$respectively. Then $$\L(\xi,\xi'):=\bigcup_{v\in [u,\eta)} \alpha_v$$
is called a \emph{bi-infinite ladder  over the ray $[u,\eta)$} corresponding to $\alpha$ $(=\alpha_u)$.
\end{defn}

Since each $\alpha_v$ is 
 coarsely defined, so is $\L(\xi,\xi')$.

\begin{lemma}\label{lem-lift in the bi infinte ladder}
	There exists $K \geq 1$ such that the following holds.
Let $\L(\xi,\xi')$ be a bi-infinite ladder over a geodesic ray $[u,\eta)\sse T$
corresponding to $\alpha \subset X_u$. Then for any $x\in\L(\xi,\xi')$ there is a $K$-qi section of $[u,\eta)$ in $X$ whose image lies in $\L(\xi,\xi')$.
\end{lemma}

\begin{proof}
Let $u=v_0,v_1,v_2\dots$ be the vertices on $[u,\eta)$ such that $d_T(u,v_n)=n$. We use the notation from Definition~\ref{def-ladder}.
Suppose $x\in \alpha_{v_n}=\L(\xi,\xi')\cap X_{v_n}$ for some vertex $v_n$ in $[u,\eta)$. Inductively we define $\sigma:[u,\eta)\to X$ as follows. Set $\sigma(v_n)=x$. There exists a uniform $K$ such that $x$ is  $K$-close to $\alpha_{v_{n\pm1}}$ in $X$. We choose $x_{n\pm1}\in \alpha_{v_{n\pm1}}$ such that $d_X(x,x_{n\pm1})\le K$ and define $\sigma(v_{n\pm1})=x_{n\pm1}$ wherever $n\pm1$ is defined. 
Since $\Pi$ is $1$-Lipschitz, $\sigma$ is a $K$-qi embedding of $[u,\eta)$. 
\end{proof}

The following result says that if we have a boundary flow of three points over a geodesic ray, then the barycenters of the flowed boundary points in vertex spaces form a qi section of that geodesic ray. 

\begin{lemma}\label{lem-barycenter flow TMS over ray}
	Let $\Pi: X \to T$ be a tree of hyperbolic spaces with 
	$T=[0,\infty)$ equipped with vertices at the integer points. 	There exists $K_{\ref{lem-barycenter flow TMS over ray}}\ge1$ depending on the parameters of the tree of spaces such that the following hold. \\
	Let $\xi_{1,0},\xi_{2,0},\xi_{3,0}\in\pa X_0$ be three distinct points such that all of them have boundary flow to $\pa X_n$ for all  $n$. Let  $\pa \Phi_{0n}(\xi_{i,0})=\xi_{i,n}$, 
	$\Xi_n=(\xi_{1,n},\xi_{2,n},\xi_{3,n})$ and  $\sigma(n)=Bary_{X_{n}}(\Xi_n)$,
	Then  $\sigma$ is a $K_{\ref{lem-barycenter flow TMS over ray}}$-qi section of $T$. 
\end{lemma}

\begin{proof}
	Since $\Pi$ is $1$-Lipschitz, it suffices to show that $d_X(\sigma(n),\sigma({n+1}))$ is uniformly bounded for all $n$. This follows  from applying Lemma \ref{lem-barycenter close} separately
	to the two qi-embeddings of the edge space $X_{n,n+1}$ into $X_{n}$ and
	$X_{n+1}$. 
\end{proof}

The next notion and the following Proposition by Kapovich--Sardar connect Cannon--Thurston maps with bi-infinite ladders. 
\begin{defn}\label{def-contractingladder}
	A bi-infinite ladder $\L(\xi_1,\xi_2)$ 
	 is said to be $(K,C)$-contracting if for all $n,m\in\Z$, there are $K$-qi sections, $\gm_{n}$, $\gm_m$ say, of a segment $[u,v]\sse [u,\eta)$ through $\al(n)$ and $\al(m)$ respectively such that  $\gm_n$, $\gm_m$ are
	 contained in $\L(\xi_1,\xi_2)$ and $$d_{X_{v}}(\gm_n(v),\gm_m(v))\le C.$$
	
	We say that $\L(\xi_1,\xi_2)$ is  contracting  if it is $(K,C)$-contracting  for some $K\ge1$ and $C\ge0$.
\end{defn}

\begin{prop}\textup{(\cite[Propositions 8.56,~8.63]{ps-kap})}\label{prop-flows over unique ray}
	With the setup and notation of Lemma \ref{lem-lift in the bi infinte ladder}, 
	assume further that $X$ is hyperbolic.
	Let $u\in V(T)$. Let $\pa i:\pa X_u\to\pa X$ be  the Cannon--Thurston map
	furnished by Theorem~\ref{thm-existence CT TMS}. Let $\xi_1\in\pa X_u$ be such that $\pa i^{-1}(\pa i(\xi_1))$ has more than one point. Then there exists a unique geodesic ray $[u,\eta)\sse T$, with $\eta\in\pa T$ satisfying the following. For any $\xi_1\ne \xi_2\in\pa i^{-1}(\pa i(\xi_1))$, a
	contracting ladder $\L(\xi_1,\xi_2)$ over $[u,\eta)$ exists.
	Further, if  $\sigma: [u,\eta) \to \L(\xi_1,\xi_2)$ is a qi section, then its 
	ideal end-point $\sigma (\infty)$ equals
 $\pa i(\xi_1)$.
\end{prop}

We now relate the qi section obtained in Lemma \ref{lem-barycenter flow TMS over ray} to Cannon--Thurston maps.

\begin{lemma}\label{lem-barycenter flow coincide}
	We continue with the setup and notation of Lemma \ref{lem-barycenter flow TMS over ray}. 
	Suppose that $X$ is hyperbolic. Let $\sigma$ denote the qi-section furnished by
	Lemma \ref{lem-barycenter flow TMS over ray}. Let $\sigma(\infty)$ denote its ideal end-point in $\pa X$.
	Let $\pa i:\pa X_0\to\pa X$ be the Cannon--Thurston map furnished by Theorem~\ref{thm-existence CT TMS}.
	Suppose further that $\pa i(\xi_{1,0})=\pa i(\xi_{2,0})$. Then $\sigma(\infty)=\pa i(\xi_{1,0})$.
\end{lemma}

\begin{proof}
We will follow the notation as in Lemma \ref{lem-barycenter flow TMS over ray}.
Let $\alpha_{ij} \subset X_0$ denote bi-infinite geodesics joining
$\xi_{i,0}$ and $\xi_{j,0}$ for $i \neq j \in \{1,2,3\}$. Then $\sigma(n)$ 
is a coarse barycenter of $\{\xi_{j,n}=\partial \Phi_{0n} (\xi_{j,0})\}_{j=1,2,3}$. In particular, there exists $K_0$ depending only on the parameters of the
tree of spaces $\Pi:X \to T$ such that for all $n$, $\sigma'(n)\in \L(\xi_{1,0},
\xi_{2,0}) \cap X_n$ and $d(\sigma(n), \sigma'(n)) \leq K_0$. Since 
$\sigma$ is a $K_{\ref{lem-barycenter flow TMS over ray}}$-qi section by
Lemma~\ref{lem-barycenter flow TMS over ray}.
 $\sigma'$ is a $(K+K_{\ref{lem-barycenter flow TMS over ray}})$-qi section. By the last statement  of Proposition \ref{prop-flows over unique ray}, $\sigma'(\infty)=\pa i(\xi_{1,0})$. Since $d(\sigma(n), \sigma'(n)) \leq K_0$ for all $n$, $\sigma(\infty)=\pa i(\xi_{1,0})$.
\end{proof}

We finally establish a reduction lemma that shows that in order to study  multiple-point preimages under  Cannon--Thurston maps, it suffices 
  to deal with a ray of spaces rather than a general tree of hyperbolic spaces.
Let $\Pi: X \to T$ be a  tree of hyperbolic spaces as in Definition~\ref{defn-tree-of-sps}. Assume that $X$ is hyperbolic. Fix a vertex $u$ of $T$. Let $\partial i_{X_u,X}: \pa X_u \to \pa X$ denote the Cannon--Thurston map
(Theorem~\ref{thm-existence CT TMS}). For any $\eta \in \partial T$, let $X_{\eta}=\Pi^{-1}([u,\eta))$.

\begin{lemma}\label{lem-enough to count over ray} With notation as above, 
Let $\xi\in\pa i_{X_u,X}(\pa X_u)$ be such  that $\xi_1,\xi_2\in\pa i^{-1}_{X_u,X}(\xi)$ are distinct points, i.e.\ $\xi$ is a multiple value of $\pa i_{X_u,X}$. Then there exists a unique $\eta \in \partial T$ such that 
$$\pa i^{-1}_{X_u,X_{\eta}}(\pa i_{X_u,X_{\eta}}(\xi_1))=\pa i^{-1}_{X_u,X}(\xi).$$
\end{lemma}

\begin{proof}
Proposition \ref{prop-flows over unique ray} furnishes a unique 
$\eta \in \partial T$, a ray $[u,\eta)\sse T$, and a contracting ladder $\L(\xi_1,\xi_2)$ over $[u,\eta)$.  Then $\L(\xi_1,\xi_2) \subset
X_{\eta} \subset X$ is a contracting ladder both in $X_{\eta}$ and in $X$. 
By Lemma \ref{lem-lift in the bi infinte ladder}, there exists $K \geq 1$ such that
the following holds. For any $p \in \L(\xi_1,\xi_2)$, there exists a $K-$qi-section $\sigma_p: [u,\eta)
\to \L(\xi_1,\xi_2)$ passing through $p$ such that $\sigma_p (\infty) =
\pa i_{X_u,X} (\xi_1) = \pa i_{X_u,X} (\xi_2)= \xi$.
Note that $\sigma_p ([u,\eta)) \subset \L(\xi_1,\xi_2) \subset X_{\eta} \subset X$.
Thus, $\sigma_p ([u,\eta))$ is a quasigeodesic ray in $X_{\eta}$ as well as in $X$.
Further, since $\L(\xi_1,\xi_2) \subset
X_{\eta} \subset X$, it is, in particular  a contracting ladder  in $X_{\eta}$.
Hence, by Proposition \ref{prop-flows over unique ray} 
$\pa i_{X_u,X_{\eta}}(\xi_1)=\pa i_{X_u,X_{\eta}}(\xi_2)$.

Next, by Theorem \ref{thm-existence CT TMS}, $X_{\eta}$ is hyperbolic, and
the inclusions $X_u\to X_{\eta}$ and $X_{\eta}\to X$ admit Cannon--Thurston maps.
Denote these by $\partial i_{X_u,X_\eta}, \partial i_{X_\eta,X}$ respectively, so that $\partial i_{X_u,X} = \partial i_{X_\eta,X} \circ\partial i_{X_u,X_\eta}$ by Lemma~\ref{lem-functoriality of CT map}.
Therefore, if $\pa i_{X_u,X_{\eta}}(\xi_1)=\pa i_{X_u,X_{\eta}}(\xi_2)$, then 
$\pa i_{X_u,X}(\xi_1)=\pa i_{X_u,X}(\xi_2)$.

Since $\xi_2$ was arbitrary, the conclusion follows.
\end{proof}

\subsection{Finite fibers of Cannon--Thurston maps for trees of spaces}
We are now in a position to prove the main theorem of the paper.

\begin{theorem}\label{thm-finite to one TMS ray}
	Given $\dl_0\ge0$, $L_0\ge0$, $D_0\ge0$ and a proper map $f:\N\to\N$ there is a constant $N_{\ref{thm-finite to one TMS ray}}=N_{\ref{thm-finite to one TMS ray}}(\dl_0,L_0,D_0,f)\in\N$ satisfying the following. 
	
	\noindent (I) Let $\Pi:X\to[0,\infty)$ be a tree of spaces
    such that the vertices in $[0,\infty)$ correspond to the non-negative integers, edges correspond to $[n,n+1]$.
    The vertex spaces $X_n$ and edge space $X_{n,n+1}$ further
    satisfy the following.
	
	\begin{enumerate}
		\item 
		Each vertex and edge space  is  a graph with valence bounded by 
		 $D_0$.	
		
		\item Vertex and edge spaces are $\dl_0$-hyperbolic.
		
		\item Edge space to vertex space maps are $L_0$-qi embeddings.
		
		\item The total space $X$ is $\dl_0$-hyperbolic.
	\end{enumerate}
	Let $i: X_0 \to X$ denote the inclusion map.
	Let $\pa i:\pa X_0\to\pa X$ denote the Cannon--Thurston map for the pair
	$(X_0,X)$ (Theorem~\ref{thm-existence CT TMS}). Then
 for all $\xi\in\pa X_0$, $$|\pa i^{-1}(\pa i(\xi))|\le N_{\ref{thm-finite to one TMS ray}}.$$
 
 	\noindent (II) More generally, let $\Pi:X\to T$ be a tree of spaces satisfying the
 	four conditions above. Then the same conclusion holds.
\end{theorem}

\begin{proof} We first prove (I) and use it to prove (II)
(see Remark~\ref{rmk-clar} below where we explicate the reason behind this structure).
	
It suffices to prove the existence of an integer $N$ such that given any finite set $A\sse \pa i^{-1}(\pa i(\xi))$, $|A|\le N$. 
If $\pa i^{-1}(\pa i(\xi))$ consists of a single point or two points, there is nothing to prove. Suppose $A$ has at least three points. By Proposition \ref{prop-flows over unique ray}, all points in $A$ have a boundary flow to $\pa X_n$ for all vertices $n$.
Given $\xi_i \in A$, let $\pa \Phi_{0n} (\xi_i) \in \pa X_n$ denote its boundary flow.
Given any triple $\xi_1 \neq \xi_2 \neq \xi_3 \in A$,
define $\sigma (\xi_1,\xi_2,\xi_3,n) = Bary_{X_n}(\pa \Phi_{0n} (\xi_1),  \pa \Phi_{0n} (\xi_2), \pa \Phi_{0n} (\xi_3))$. 
Then each such sequence $\{\sigma (\xi_1,\xi_2,\xi_3,n)\}_n$
is a $K_{\ref{lem-barycenter flow TMS over ray}}$-qi section of $[0,\infty)$ into $X$ by Lemma \ref{lem-barycenter flow TMS over ray}. Let $\sigma (\xi_1,\xi_2,\xi_3,\infty)$
denote the unique accumulation point of $\{\sigma (\xi_1,\xi_2,\xi_3,n)\}_n$ in $\partial X$.
 By Lemma \ref{lem-barycenter flow coincide}, $\sigma (\xi_1,\xi_2,\xi_3,\infty)=
 \pa i(\xi_1)= \pa i(\xi_2)= \pa i(\xi_3)$.

By Lemma \ref{lem-lifts imp close in fiber} it follows that for all large $n\in\N$, 
there exists $R=R_{\ref{lem-lifts imp close in fiber}}(K_{\ref{lem-barycenter flow TMS over ray}})$ such that $$d_{X_n} (\sigma (\xi_1,\xi_2,\xi_3,n), \sigma (\xi'_1,\xi'_2,\xi'_3,n)) \leq R,$$
for any pair of distinct triples $\{\xi_1 \neq \xi_2 \neq \xi_3 \}$ and  $\{\xi_1' \neq \xi_2' \neq \xi_3' \}$ in $A$. Thus, there exists $x_n \in X_n$ such that 
any such $\sigma (\xi_1,\xi_2,\xi_3,n)$ lies in an $R-$ball about $x_n$ in $(X_n, d_{X_n})$.

	Fix one such large enough $n_0\in\N$. Let $A_{n_0}$ denote the set of  points
	in $\partial X_n$ obtained as a boundary flow of points of $A$ to $\pa X_{n_0}$. Then $A_{n_0}$ has the property that for all triple $\beta=(\beta_1,\beta_2,\beta_3)$ of distinct points in $A_{n_0}$, $Bary_{X_{n_0}}(\beta)\cap B^{X_{n_0}}(x_{n_0},R)\ne\emptyset$. Finally, Proposition \ref{prop-bdd nm barycenters in bdd ball} shows that the number of elements in $A_{n_0}$, and hence in $A$, is bounded by a constant $D'$ depending only on $\dl_0,D_0,R$. 
	Choosing $N=\max\{D',2\}$ completes the proof of (I).
	
	Next, we consider the case of a general tree of spaces. By Theorem~\ref{thm-existence CT TMS} we have the following.
	For any $\eta \in \partial T$ and base vertex $u \in T$, 
	$X_{\eta}:=\Pi^{-1}([u, \eta))$ is hyperbolic with respect to the induced path metric from $X$ (see Theorem \ref{thm-existence CT TMS}). Further, the hyperbolicity constant of $X_{\eta}$
    depends only on $\delta_0$. Without loss of generality, we assume that each $X_{\eta}$ is
	$\delta_0$ hyperbolic, so that 
	each $X_{\eta}$ satisfies condition (4) in the hypothesis of part (I). Since $X$ satisfies the first three conditions, so does each $X_{\eta}$. For each $X_{\eta}$, let $i_\eta: X_0 \to X_{\eta}$ denote the inclusion map and 
	$\partial i_\eta$ the associated Cannon--Thurston map (Theorem~\ref{thm-existence CT TMS}).
	Part (I) now gives a uniform bound $N$, independent of $\eta$ on the cardinality of
	$\pa i_\eta^{-1}(\pa i_\eta(\xi))$ for $\xi \in \partial X_0$.
	Finally,  Lemma \ref{lem-enough to count over ray} shows that for each
	$\xi \in \partial X_0$, there exists $\eta \in \partial T$ such that
$\pa i^{-1} (\pa i(\xi)) = \pa i_\eta^{-1}(\pa i_\eta(\xi))$ 
concluding the proof.
	\end{proof}
	
	\begin{rmk}\label{rmk-clar}
	    We have divided the statement of the above theorem into two parts so that the first part can be referenced directly in the proof of Theorem~\ref{thm-finite to one MB gen}.
	\end{rmk}

As an immediate consequence of Theorem \ref{thm-finite to one TMS ray}, we have the following group-theoretic result.

\begin{cor}\label{thm-gog vertex}
	Suppose $G$ is a hyperbolic group which can be decomposed as a finite graph of hyperbolic groups with the qi embedded condition and $H$ is a vertex group. Then the Cannon--Thurston map $\pa i:\pa H\to\pa G$ is uniformly finite-to-one.
\end{cor}

\section{Metric graph bundles}\label{sec-mgb}
Metric graph bundles were introduced in \cite{pranab-mahan} as a coarse-geometric generalization of  fiber bundles. Examples include trees of metric spaces where
each edge-to-vertex space map is a uniform quasi-isometry (and not just a qi embedding).

\subsection{Preliminaries on Metric graph bundles}\label{sec-pmgb}
In this subsection, we recall the material from 
\cite{pranab-mahan} that we will need subsequently.

\begin{defn}\textup{(\cite[Definition 1.5]{pranab-mahan})}\label{defn-metric graph bundle}
	Let $X$ and $B$ be  graphs equipped with their natural metrics $d_X, d_B$. Let $f:\N\map\N$ be a proper map, i.e. $f(n)\map\infty$ as $n\map\infty$. We say that $X$ is an $f$-metric graph bundle over $B$ if there is a simplicial, surjective (and $1$-Lipschitz) map $p:X\map B$ such that the following hold.
	\begin{enumerate}
		\item For all $b\in V(B)$, $F_b:=p^{-1}(b)$, called the fiber over $b$, is a connected subgraph of $X$. The intrinsic path metric on $F_b$ is denoted as
		$d_b$. The inclusion maps $(F_b,d_b)\map X$ are $f$-proper embedding.
		
		\item Let $b_1,b_2\in V(B)$ be adjacent vertices, and let $x\in V(F_{b_1})$. Then  $x$ is connected by an edge in $X$ to a point in $V(F_{b_2})$.
	\end{enumerate}
\end{defn}

We shall often simply say that $p:X\map B$ is an $f$-metric graph bundle to summarize the content of Definition~\ref{defn-metric graph bundle}.

\begin{eg}\label{eg-mgb} Examples include
\begin{enumerate}
\item finitely generated normal subgroups $H$ of finitely generated groups $G$.
Choosing a finite generating set of $H$, extending to one for $G$ and inducing a
finite generating set of $Q=G/H$, we have that $p: \Gamma_G \to \Gamma_Q$ is a metric graph bundle with fibers $G$-translates of $\Gamma_H$ (see \cite[Example $1.8$]{pranab-mahan}).
\item and more generally, finitely generated commensurated subgroups $H$ of finitely generated groups $G$. Here $\Gamma_Q$ is replaced by the  the graph whose vertices are the cosets $gH$ and edges join pairs $\{gH,gsH\}$, where $s$ ranges over the finite generating set of $G$ used in constructing $\Gamma_G$.
(See \cite[Proposition 3.14]{margolis-almostnormal} and \cite[Proposition 5.12]{NirMaMj-commen}. Note that this graph differs from the Schreier graph whose edges join pairs $\{gH,sgH\}$. This graph is in fact quasi-isometric to the graph obtained from $\Gamma_G$ by electrifying the left cosets of $\Gamma_H$.)
\end{enumerate}
\end{eg}

The following lemma guarantees that adjacent fibers are quasi-isometric via a natural map.
	\begin{lemma}\textup{(\cite[Proposition 1.7]{pranab-mahan})}\label{lem-natural map QI}
Let $b_1,b_2\in V(B)$ such that $d_B(b_1,b_2)=1$. Let $\phi:F_{b_1}\map F_{b_2}$ be any map that sends $x\in V(F_{b_1})$ to $\phi(x)\in V(F_{b_2})$ such that $x$ and $\phi(x)$ are joined by an edge in $X$. Then $\phi$ is a $k_{\ref{lem-natural map QI}}(f)$-quasiisometry where $k_{\ref{lem-natural map QI}}(f)\ge1$ is a constant depending only on the function $f$.
\end{lemma}
We now introduce hyperbolicity.
\begin{defn}\label{defn-controlled fibers}
	We say that an $f$-metric graph bundle $p:X\to B$ has controlled hyperbolic fibers if for all $b\in V(B)$, $F_b$ is a $\dl_0$-hyperbolic graph  and the barycenter map $\pa^3F_b\to F_b$ is $L'$-coarsely surjective for some fixed $\dl_0\ge0$ and $L'\ge0$.
\end{defn}

\noindent{\bf Convention.} We assume henceforth in this section that $p:X\to B$ is an $f$-metric graph bundle with controlled hyperbolic fibers with the parameters $\dl_0$, $L'$ as in Definition \ref{defn-controlled fibers}. Moreover, $X$ is $\dl_0$-hyperbolic and by \textup{(\cite[Proposition $2.10$]{pranab-mahan})}, $B$ is hyperbolic. We also assume that $B$ is $\dl_0$-hyperbolic. We denote the fiber of the metric bundle over a vertex $v$ in $B$ by $F_v$.\smallskip

In \cite{pranab-mahan}, it was shown that the total space $X$ of an $f$-metric graph bundle 
$p:X \to B$ with controlled hyperbolic fibers is  hyperbolic if and only if the $p:X \to B$ satisfies a Bestvina--Feighn  flaring condition and $B$ is hyperbolic.
We shall need the following consequence (see \cite[Theorem 4.3, Proposition 5.8 and Remark 4.4]{pranab-mahan}).

\begin{theorem}\label{thm-hyp MB}
Let $p:X\to B$ be the metric graph bundle with parameters $\dl_0$, $L'$. Then, given $\dl_0\geq 0, K\geq 1$,
there exists $\dl_1 \geq 0$ such that the following holds.
If $X$ is  $\dl_0$ hyperbolic, then for any $K-$qi embedded subspace $A\sse B$, $p^{-1}(A)$ is $\dl_1$-hyperbolic with its intrinsic path metric.
\end{theorem}

The following theorem ensures the existence of Cannon--Thurston maps in the context of metric graph bundles.

\begin{theorem}\label{thm-existence CT MB}
Suppose $p:X\to B$ is an $f$-metric graph bundle with controlled hyperbolic fibers. Assume that $X$ is hyperbolic.

\begin{enumerate}
\item \textup{(\cite[Theorem $5.3$]{pranab-mahan})}For any vertex $b$ of $B$, the inclusion $i_b: F_b\to X$ admits a Cannon--Thurston map.

\item \textup{(\cite[Theorem $5.2$]{ps-krishna})} For any qi embedded subgraph $A$ in $B$, the inclusion $p^{-1}(A)\to X$ admits a Cannon--Thurston map.
\end{enumerate}
\end{theorem}

\begin{defn}[Bi-infinite Ladders]\label{def-ladder-mgb} For 
 Let $b\in V(B)$, $\xi_1,\xi_2\in\pa F_b$. Let $\al$ is a geodesic ray joining $b$ and $\eta\in\pa B$. Let $b=b_0, \cdots, b_n, \cdots$ be the sequence of vertices on $\alpha$. Let $\phi_n$ denote the quasi-isometry obtained by composing the quasi-isometries between $F_{b_i} , F_{b_{i+1}}$  (cf.\ Lemma \ref{lem-natural map QI}) in order for $i=0, \cdots, n-1$. Let 
 $l_n$ denote a bi-infinite geodesic  joining $\pa\phi_n(\xi_1)$ and $\pa\phi_n(\xi_2)$ in $F_{b_n}$.
 Then $$\L_\eta(\xi_1,\xi_2):=\bigcup_{n\in\{0\}\cup\N}l_n$$
 is called the \emph{bi-infinite ladder} corresponding to $l_0$ along $\alpha$.
 \end{defn}

 \begin{rmk} We remark here that the notion of a bi-infinite ladder above and the  structure of the subsequent exposition below is quite similar to that in Section~\ref{sec-treebdls}. Note that in Section~\ref{sec-treebdls}, the vertex and edge spaces are not required to be quasi-isometric to each other, whereas the base space is a tree.
 On the other hand, in a metric graph bundle, 
 the vertex and edge spaces are necessarily uniformly quasi-isometric to each other, whereas the base is an arbitrary graph of bounded valence. This difference necessitates a bit of repetition in the exposition below.
 \end{rmk}

 It is easy to see that if $\alpha, \alpha'$ are two geodesics (or even quasi geodesics) in $(B,d_B)$ joining $b, \eta$, then the bi-infinite ladders corresponding to $l_0$ along $\alpha, \alpha'$ are at a bounded Hausdorff distance from each other.
 Thus, $\L_\eta(\xi_1,\xi_2)$ is coarsely well-defined once $\eta, \xi_1, \xi_2$ are given.

The notion of a contracting ladder (Definition~\ref{def-contractingladder}) now goes through without any change and is defined as in the context of metric graph bundles.
We shall need the following analog of Proposition \ref{prop-flows over unique ray}.

\begin{prop}\cite[Lemma $6.17$]{ps-krishna}\label{prop-unique ray in base}
Let $p:X\to B$ be an $f$-metric graph bundle with controlled hyperbolic fibers. Assume that $X$ is hyperbolic.

Let $b\in V(B)$. Let $\pa i_{F_b,X}:\pa F_b\map\pa X$ be the Cannon--Thurston map
from Theorem~\ref{thm-existence CT MB}. Let $z\in\pa X$ be such that $|\pa i_{F_b,X}^{-1}(z)|\ge2$. Then there is a point $\eta\in\pa B$ with the following property. Fix a geodesic ray $\al$ joining $b$ and $\eta$. For any $\xi_1,\xi_2\in\pa i^{-1}_{F_b,X}(z)$, the ladder $\L(\xi_1,\xi_2)$ over $\al$ is contracting. Moreover, for any qi section, $\widetilde{\al}$  of $\al$ contained in $\L(\xi_1,\xi_2)$, we have $\widetilde{\al}(\infty)=z$. In particular, $\eta$ is determined by $z$.
\end{prop}

We collect together a few observations that we will need below. These are analogs of the corresponding statements in Section~\ref{sec-bflow}.
\begin{lemma}\label{lem-omn} We continue with the setup and notation of Definition~\ref{def-ladder-mgb}. 
\begin{enumerate}
\item Given any point $w$ in a bi-infinite ladder $\L(\xi_1,\xi_2)$ over $\al$, there is a qi section $\widetilde{\al}$ of $\al$ through $w$ such that $\widetilde{\al} \subset \L(\xi_1,\xi_2)$ (see Lemma \ref{lem-lift in the bi infinte ladder} for instance, the argument is unchanged in the present context).
\item $X_{\al}=p^{-1}(\al)$ is hyperbolic (Theorem \ref{thm-hyp MB}).
\item Let $\L(\xi_1,\xi_2)$ be as in Proposition~\ref{prop-unique ray in base}.
Then $\L(\xi_1,\xi_2)\subset X_{\al} \subset X$ is also a contracting ladder in $ X_{\al}$. 
\item The inclusions $F_b\to X_{\al}$, and $X_{\al}\to X$ admit Cannon--Thurston maps by Theorem \ref{thm-existence CT MB}.
\end{enumerate}
\end{lemma}

\begin{rmk}\label{rmk-mgbray=tree} We note here that
	$p: X_{\al} \to \al$ is quasi-isometric to a tree of metric spaces as follows. The vertices are given by the vertices $b_i$ along $\alpha$ and vertex spaces are $F_{b_i}$. The edge space corresponding to the edge $[b_i,b_{i+1}]$
	is also $F_{b_i}$, so that $F_{b_i} \times \{b_i\}$ is glued to $F_{b_i}$ by identity on the first factor, and $F_{b_i} \times \{b_{i+1}\}$ is glued to $F_{b_{i+1}}$ by the uniform quasi-isometry from Lemma~\ref{lem-natural map QI}.
\end{rmk}

 The following is an analog of Lemma \ref{lem-enough to count over ray}. With Lemma~\ref{lem-omn} in place, the proof is a replica and we omit it.

\begin{cor}\label{cor-preimages are same}
We continue with the hypotheses and notation of Proposition \ref{prop-unique ray in base}. Let
$z\in \partial X$ be such that $\pa i^{-1}_{F_b,X}(z)$
has cardinality more than one. Then there exists a geodesic ray 
$\alpha:[0,\infty)\map B$  starting at $b$, with $\al(\infty)$ equal to some $\eta\in\pa B$, such that the following holds. Let $X_{\al}=\Pi^{-1}(\al)$.  Let $\xi \in \partial F_b$ be such that $\pa i_{F_b,X} (\xi) = z$. 
Then $$\pa i^{-1}_{F_b,X_{\al}}(\pa i_{F_b,X_{\al}}(\xi))=\pa i^{-1}_{F_b,X}(z).$$
\end{cor}

\subsection{Cannon--Thurston fibers For metric graph bundles}\label{sec-ctfibermgb}
We are now in a position to prove the analog of Theorem~\ref{thm-finite to one TMS ray} for metric graph bundles.

\begin{theorem}\label{thm-finite to one MB gen}
	Given $\dl_0\ge0$, $L'\ge0$, $D_0\ge0$ and a proper map $f:\N\to\N$ there is a constant $N=N(\dl_0,L',D_0,f)$ satisfying the following. Let $p:X\to B$ be an $f$-metric graph bundle with controlled fibers satisfying the following.
	
	\begin{enumerate}
		\item  {Each $F_b$ is a graph where all vertices have valence bounded by $D_0$}.
		
		\item The total space $X$ is $\dl_0$-hyperbolic.
	\end{enumerate}

	Then for any $b\in V(B)$, the Cannon--Thurston map $\pa i:\pa F_b\to\pa X$ is uniformly finite-to-one, i.e.\  for all $\xi\in \pa X$, $$ |\pa i^{-1}(\xi)|\le N.$$
\end{theorem}

\begin{proof} 
For the upper bound, we first deal with the case that $B$ is a geodesic ray $[0,\infty)$ where $[0,\infty)$ has the standard graph structure with vertices at the integer points.
 By Remark~\ref{rmk-mgbray=tree}, $p:X\to[0,\infty)$  is quasi-isometric to a tree of spaces, whose vertex spaces are exactly those of $p:X\to[0,\infty)$.
 This special case is then an immediate consequence of Theorem \ref{thm-finite to one TMS ray}(I).
 
 We now deal with the upper bound in the general case. The argument is essentially the same as that for Theorem \ref{thm-finite to one TMS ray}(II). We include a sketch for completeness.

Assume that $\pa i^{-1}(\xi)$ contains more than two elements. Also, let $\xi_1 \in \pa F_b$ be such that $\pa i(\xi_1) = \xi$.
 Let $\eta \in \partial B$ be the boundary point of $B$ determined by $\xi$ as per the last statement of Proposition~\ref{prop-unique ray in base}.
By Corollary \ref{cor-preimages are same}, it suffices to consider a metric bundle $p: X_\eta \to \al$ over a geodesic ray $\al \subset B$, where $\alpha (\infty) =\eta$.  By Theorem \ref{thm-hyp MB}, $X_{\al}:=p^{-1}(\al)$ is $\dl_1$-hyperbolic for some $\dl_1$ depending only on $\dl_0$, $L'$ and $f$. Thus the restricted metric graph bundle $p:X_{\al}\to\al\sse B$ over the geodesic ray $\al$ satisfies all the conditions of the present theorem  with constants $\dl_1$, $D_0$ and $f$. 
By the special case dealt with at the beginning of the proof, the cardinality of the set $\pa i^{-1}_{F_b,X_{\al}}(\pa i_{F_b,X_{\al}} (\xi_1))$ is bounded by an integer $N = N(\dl_1,D_0,f)$.
Since $\dl_1$ depends only on $\dl_0, L',f$ the theorem follows.
\end{proof}

\begin{rmk} Since the Cannon--Thurston map $\pa i$ in the above theorem is surjective by \cite{halder-surj-CT}, it follows that
    $1 \leq  |\pa i^{-1}(\xi)|\le N.$
\end{rmk}

As a consequence of Theorem \ref{thm-finite to one MB gen} we immediately conclude the following
(see Example~\ref{eg-mgb}).

\begin{cor}\label{cor-normal}
Let $H$ be a nonelementary hyperbolic normal or commensurated subgroup of infinite index in a hyperbolic group $G$. Then the Cannon--Thurston map $\pa i:\pa H\to\pa G$ is uniformly finite-to-one.
\end{cor}

\section{Applications and generalizations}\label{sec-further extension}
\subsection{New proofs of old theorems}\label{sec-newproofs}

We give a new and short proof of \cite[Theorem B]{NirMaMj-commen}.

\begin{theorem}\label{thm-lmm}
Let $p: X \to [0, \infty)$ be a metric graph bundle with controlled hyperbolic fiber $F$, where $F$ is quasi-isometric to a one-ended hyperbolic group $H$.
Then $\partial F$  admits a local cut-point. Hence,   $H$  virtually splits over a 2-ended subgroup.
\end{theorem}

\begin{proof}
	The boundary $\partial X$ is a dendrite \cite{bowditch-cutpts}. Let $i: F \to X$ denote the inclusion map, and
$\partial i: \partial F \to \partial X$ denote the associated Cannon--Thurston map. Since $H$ is one-ended, $\partial F$ is a compact, connected, locally connected set. Hence, so is
$\partial i (\partial F)$. In fact, $\partial i (\partial F)$ is
equal to $\pa X$ \cite{halder-surj-CT}.
 For any cut-point $\xi$ of 
$\pa X = \partial i (\partial F)$, $\partial i^{-1} (\xi)$  disconnects $\partial F$. By Theorem~\ref{thm-finite to one MB gen}, $\partial i^{-1} (\xi)$ is finite.
Hence 
$\partial F$  admits a local cut-point.  The last statement now follows from Bowditch's theorem on JSJ decomposition of hyperbolic groups \cite{bowditch-cutpts}.
\end{proof}

The notion of
3-manifolds that slither over the circle \cite{thurston-slither1,thurston-slither2,calegari-dunfield,calegari-zippers}
can be translated 
to the context of metric graph bundles. We use this to recover Thurston's theorem that the associated Cannon-Thurston map is finite-to-one. 

We refer the reader to \cite[Section 4]{calegari-zippers} for details on uniform quasimorphisms and 
\cite[Example 4.12]{calegari-zippers} for an elucidation  of the fact that 3-manifolds that slither in the sense of 
\cite{thurston-slither1,thurston-slither2} give examples of
 uniform quasimorphisms.  In fact the explanation in \cite[Example 4.12]{calegari-zippers} shows a bit more.
 Let $M$ be a closed hyperbolic 3-manifold slithering over the circle, and let $G=\pi_1(M)$. Let $\phi: G \to \Z$ denote the associated uniform quasimorphism.
 Let $\Gamma$ denote a Cayley graph of $G$ with respect to a finite generating set.
 Then there exists a projection $\Pi: \Gamma \to \Z$
 such that (after possibly adding some edges to $\Gamma$ while retaining its quasi-isometry type) $\Pi^{-1}(n)$ is a connected graph. Further,
 by Candel's theorem \cite{candel} $\Pi^{-1}(n)$ is uniformly quasi-isometric to the hyperbolic plane. Thus, 
 $\Pi: \Gamma \to \Z$ gives rise to a metric graph bundle
 in the sense of Definition~\ref{defn-metric graph bundle} with controlled hyperbolic fibers.
 In fact the projection map $\Pi$ in the metric graph bundle is quasi-equivariant with respect to the 
 action by $G$ on the total space and the quasi-action by $G$ on the base $\Z$
 via the quasimorphism $\phi$.

Let $\mathcal L$ denote any leaf of the pulled back foliation on $\til M$ and we assume by Candel's theorem \cite{candel}
that $\mathcal L$ is isometric to the hyperbolic plane.
 It now follows from Theorem~\ref{thm-existence CT MB} that 
 $i: \mathcal L \to \til{M}$ admits a Cannon--Thurston map. Further,
 by Theorem~\ref{thm-finite to one MB gen} this map is finite-to-one.

 We also note that Theorem~\ref{thm-lmm} shows that there cannot exist  analogs of Thurston's slithering examples in higher dimensions in the following sense. Let $M$ be a closed manifold of pinched negative curvature
 of dimension greater than 3, and let $\mathcal{F}$ be a codimension one uniform
  foliation in the sense of Thurston, i.e.\ $\Pi: \til M \to \R$ is a metric
   bundle in the sense of \cite{pranab-mahan}.
    Let $\til{\mathcal{F}}$ denote the
    pulled back foliation to
 $\til M$.
 Then no (simply connected)  leaf $\mathcal L$ of $\til{\mathcal{F}}$ can itself be 
  of pinched negative curvature 
 with respect to the metric induced from $\til M$. Else, the boundary $\pa{\mathcal{L}}$ is a sphere  of dimension greater than one, and hence cannot have local cut-points, violating Theorem~\ref{thm-lmm}.
This generalizes the well-known fact that a closed negatively curved manifold of dimension greater than 3 cannot fiber over  the circle with fibers of pinched negative curvature.

\subsection{Subtrees of spaces}\label{subsec-subtree of spaces}
Let $\Pi:X\to T$ be a tree of hyperbolic metric spaces satisfying the qi embedded condition. Further, assume that $X$ is hyperbolic. Let $S\sse T$ be a connected subtree of $T$ and $Y=\Pi^{-1}(S)$. Then  $Y$ is hyperbolic  and the inclusion $Y\to X$ admits a Cannon--Thurston map $\pa i_{Y,X}:\pa Y\to\pa X$ (Theorem~\ref{thm-existence CT TMS}).
In this subsection we generalize Theorem \ref{thm-finite to one TMS ray} as  follows.

\begin{theorem}\label{thm-subtree of spaces}
The Cannon--Thurston map $\pa i_{Y,X}:\pa Y\to\pa X$ is uniformly finite-to-one. In fact, for all $\xi\in\pa i_{Y,X}(\pa Y)$, $$|\pa i^{-1}_{Y,X}(\xi)|\le N$$ where $N$ is the constant in Theorem \ref{thm-finite to one TMS ray}.
\end{theorem}

We start with the following proposition due to Kapovich--Sardar. It allows us a key reduction to the case already dealt with in Theorem \ref{thm-finite to one TMS ray}.

\begin{prop}\textup{(\cite[Proposition $8.49$]{ps-kap})}\label{prop-geo line in a single vertex}
	There exists $R\ge0$ such that the following holds.
	Let  $\eta_1 \neq \eta_2\in \pa Y$ be such that $\pa i_{Y,X} (\eta_1)=\pa i_{Y,X} (\eta_2)$.
	Let
	$\al$ be a bi-infinite geodesic in $Y$ joining $\eta_1,\eta_2$. Then there exists a vertex $u$ of $S$ and a bi-infinite geodesic $\bt$ in $X_u$ such that $\al,\bt$ lie within an $R-$neighborhood of each other.
\end{prop}

{Proposition \ref{prop-geo line in a single vertex} roughly says that multiple points of the map
$\pa i_{Y,X}$ are witnessed in vertex spaces}. We will use induction on 
the cardinality of $\pa i_{Y,X}^{-1} (\xi)$ to obtain a version of Proposition \ref{prop-geo line in a single vertex} for finitely many $\eta_i$'s instead of 2. Towards this, we need the following.

\begin{prop}\textup{(\cite[Proposition 2.3]{ps-conical-c})}\label{prop-bdry flow in beween}
	Let $\al$ and $\bt$ be geodesic rays in $X_u$ and $X_v$ respectively such that $\pa i_{X_u,X}(\al(\infty))=\pa i_{X_v,X}(\bt(\infty))$. Then there exists a vertex $w\in[u,v]$ such that both $\al(\infty)$ and $\bt(\infty)$ admit boundary flow to $\pa X_w$.
	(Here, $[u,v] \subset T$ is the geodesic joining $u, v$.)
\end{prop}

We are now in a position to prove the following generalization of Proposition \ref{prop-geo line in a single vertex}.

\begin{lemma}\label{lem-pull in a single vertex}
Suppose $A$ is a finite subset of $\pa i^{-1}_{Y,X}(\xi)$ for some $\xi\in\pa i_{Y,X}(\pa Y)$. Then there is a vertex $w$ of $S$ such that  $A$ is contained in the image of $\partial i_{X_w,Y}$.
\end{lemma}

\begin{proof}
We  induct on $|A|$, the cardinality of $A$. If $|A| = 2$, then we are done by Proposition \ref{prop-geo line in a single vertex}. 
Suppose that the statement is true when $|A| \leq n$. Let $|A| = n+1$. Let $\{\xi_1, \xi_2 \cdots \xi_{n+1}\}=A$. 

By the induction hypothesis, there exist $u,v\in V(S)$ such that $\xi_1, \xi_2, \cdots, \xi_n \in \partial X_v$ and $\xi_2, \xi_3 \cdots \xi_{n+1} \in \partial X_u$. This means that for all vertices $w$ in $[u,v]$, all the points $\xi_2,\xi_3,\cdots,\xi_n$ admit 
a boundary flow to $\pa X_w$. Further, by Proposition \ref{prop-bdry flow in beween}, we can find a vertex $w$ in $[u,v]$ such that both $\xi_1, \xi_{n+1}$ 
admit 
a boundary flow to $\pa X_w$. Hence all points $\xi_1,\xi_2,\cdots,\xi_{n+1}$ admit 
a boundary flow to $\pa X_w$.  This completes the proof.
\end{proof}

\begin{proof}[Proof of Theorem \ref{thm-subtree of spaces}]
Let $A\subseteq \partial i_{Y,X}^{-1}(\xi)$ be a finite set, then by Lemma \ref{lem-pull in a single vertex} there exists a vertex $u$ of $S$ such that $A$ is contained in the image of $\partial i_{X_u,Y}$. Hence
$$|A|\le |\partial i_{X_u,Y}^{-1}(A)|\le | \partial i_{X_u,Y}^{-1}( \partial i_{Y,X}^{-1}(\xi))|=|\partial i_{X_u,X}^{-1}(\xi)|\le N_{\ref{thm-finite to one TMS ray}}(\dl_0,L_0,D_0,f)$$
where the equality follows from Lemma \ref{lem-functoriality of CT map} and the last inequality follows from Theorem \ref{thm-finite to one TMS ray}.
\end{proof}

The following immediate group-theoretic consequence of Theorem~\ref{thm-subtree of spaces} generalizes Corollary \ref{thm-gog vertex}.

\begin{cor}\label{cor-gog-subtree}
	Suppose $(\G,\Y)$ is a finite graph of hyperbolic groups satisfying the qi embedded condition. Assume that the fundamental group $G=\Pi_1(\G,\Y)$ is hyperbolic. Let $\Y'$ be a connected subgraph of $\Y$, and $(\G,\Y')$ be the restricted graph of groups. Let $G'=\Pi_1(\G,\Y')$. {Then $G'$ is hyperbolic and the Cannon--Thurston map $\pa i:\pa G'\to\pa G$ is uniformly finite-to-one}.
\end{cor}

\begin{remark}
Suppose $p:X\to B$ is an $f$-metric graph bundle with  controlled hyperbolic fibers and parameters $\dl_0$ and $L'$. We also assume that each fiber has vertices with valence bounded uniformly by $D_0$. Let $B'$ be a qi embedded subspace of $B$, and let $Y=p^{-1}(B')$. Note that $Y$ is uniformly hyperbolic with the induced path metric from $X$ (Theorem \ref{thm-hyp MB}) and we have the Cannon--Thurston map $\pa i_{Y,X}:\pa Y\to \pa X$ (see Theorem \ref{thm-existence CT MB} $(2)$). A similar argument to the proof of Theorem \ref{thm-finite to one MB gen} shows that the Cannon--Thurston map $\pa i_{Y,X}$ is uniformly finite-to-one. We sketch this argument below.

Let $A$ be a finite subset of $\pa i^{-1}_{Y,X}(\xi)$ and $|A|\ge2$ for some $\xi\in\pa X$. 
Now \cite[Theorem 6.25]{ps-krishna} says that given any two distinct points $\eta_1,\eta_2$ in $A$, there is a fiber $F_b$,  and points $\xi_1,\xi_2\in\pa F_b$ such that $\pa i_{F_b,Y}(\xi_i)=\eta_i$. Since adjacent fibers are (uniformly) quasi-isometric to each other (Lemma \ref{lem-natural map QI}),  all such  points $\xi_1,\xi_2$ admit boundary flows to the boundary of any fiber of $Y$. Since $\pa i_{F_b,X}=\pa i_{Y,X}\circ\pa i_{F_b,Y}$ (Lemma \ref{lem-functoriality of CT map}), the cardinality of $A$ has a uniform bound depending only on $\dl_0$, $L'$, $D_0$ and $f$ by Theorem \ref{thm-finite to one MB gen}.
\end{remark}

\subsection*{Subtrees of subspaces}
In \cite{HS-CT}, the authors study a generalization of Theorem \ref{thm-existence CT TMS}. 
We will see below that the Cannon--Thurston maps studied there are uniformly finite-to-one  as well. 
As a representative case, we present Theorem \ref{thm-CT subtree of subspaces} below, where the existence of a Cannon--Thurston map is given by \cite[Theorem C]{HS-CT}. 
Similar statements also hold for \cite[Theorems A, B and D]{HS-CT}. We refer the reader to \cite[Example 1.2]{HS-CT} for a new example to which Theorem~\ref{thm-CT subtree of subspaces} applies.

\begin{theorem}\label{thm-CT subtree of subspaces}
	Let $\Pi:X\to T$ be a tree of uniformly hyperbolic spaces satisfying the qi embedded condition such that $X$ is hyperbolic. Let $Y\sse X$. 
    Suppose that the restriction $\Pi_Y:Y\to \Pi(Y)=S\sse T$ is also a tree of uniformly hyperbolic spaces satisfying the qi embedded condition. We further assume that  for any edge $e$ of $S$ incident to a vertex $u$ of $S$, the following hold:

	\begin{enumerate}
	\item A Cannon--Thurston map $\pa i_{Y_u,X_u}:\partial Y_u\to\partial X_u$ exists (here $X_u=\Pi^{-1}(u)$, $Y_u=\Pi_Y^{-1}(u)$). Further, the Cannon--Thurston maps $\pa i_{Y_u,X_u}$  are uniformly finite-to-one, i.e.\ there exists
    $N'$ such that for all $u$, and $\xi_u \in \partial X_u$,
    $\pa i_{Y_u,X_u}^{-1} (\xi_u)$ has cardinality bounded by $N'$. 
	
	\item The inclusion of edge spaces $Y_e \to X_e$ are
    uniform qi embeddings.
	
	\item There exists $R_0 > 0$ such that the following holds.
    Let $X_{eu}:=f_{e}(X_e)$ denote the image of $X_e$ in $X_u$. Similarly, let $Y_{eu}:=f_{e}(Y_e)$.
    Let $P^{X_u}_{X_{eu}}$ (resp.\ $P^{Y_u}_{Y_{eu}}$)
	denotes a nearest-point projection of $X_u$ 
    (resp.\ ${Y_u}$)
    onto $X_{eu}$ (resp.\ ${Y_{eu}}$) in $X_u$
    (resp.\ ${Y_u}$). Then
	$$
	d_{X_u}\big(P^{X_u}_{X_{eu}}(y), P^{Y_u}_{Y_{eu}}(y)\big) \le R_0
	\text{ for any } y \in Y_u.
	$$ 
\end{enumerate}
Then $Y$ is hyperbolic and the Cannon--Thurston map $\pa i_{Y,X}:\pa Y\to\pa X$ (furnished by \cite[Theorem C]{HS-CT}) is uniformly finite-to-one.
\end{theorem}

    \begin{proof}
    We assume that both the trees of spaces $\Pi:X\to T$ and $\Pi_Y:Y\to S$ have parameters $\dl_0$, $L_0$, $f$ and $D_0$. 
    The Cannon--Thurston map $\pa i_{Y,X}$ exists by \cite{HS-CT}[Theorem C].
    For a vertex $u\in V(S)$, we have the following commutative diagram of Cannon--Thurston maps
    (cf.\ Lemma~\ref{lem-functoriality of CT map}):
\[
\begin{tikzcd}
\partial Y_u \arrow{r}\arrow{d} & \partial Y \arrow{d} \\
\partial X_u \arrow{r} & \partial X 
\end{tikzcd}
\]

Let $\xi\in\pa i_{Y,X}(\pa Y)$, and let $A=\{\eta_1,\eta_2,\dots,\eta_n\}\sse \pa i^{-1}_{Y,X}(\xi)$, $n\ge 2$. 
        As a corollary of \cite{HS-CT}[Theorem F], for each $\eta_i$ there exists a vertex $u_i\in V(S)$ and $\eta_{i,u_i}\in \pa Y_{u_i}$ such that $\eta_i =\pa i_{Y_{u_i},Y}(\eta_{i,u_i})$. 
        Let $\xi_{i,u_i} =\pa i_{Y_{u_i},X_{u_i}}(\eta_{i,u_i})$.
        As in the proof of Lemma~\ref{lem-pull in a single vertex}, there exists a single vertex $v\in V(S)$ such that $\xi_{i,u_i}\in \pa X_{u_i}$ admits a flow to $\xi_{i,v}\in \pa X_v$.
        By repeated application of \cite{HS-CT}[Lemma 2.50] in the vertex spaces and by \cite{HS-CT}[Lemma 3.33 (3)], we conclude that $\eta_{i,u_i}\in \pa Y_{u_i}$ admits a flow to $\eta_{i,v}\in \pa Y_v$. If $\al_i$ and $\bt_i$ denote geodesic rays in $Y_{u_i}$ and $Y_v$ respectively representing $\eta_{i,u_i}$ and $\eta_{i,v}$, then $Hd_Y(\al_i,\bt_i)<\infty$ (\cite{ps-conical}[Lemma 4.4]).
        Hence $\eta_i =\pa i_{Y_{u_i},Y}(\eta_{i,u_i}) = \pa i_{Y_v,Y}(\eta_{i,v})$ by \cite{ps-krishna}[Lemma 2.55].
        In particular, $A\subseteq \pa i_{Y_v,Y}(\pa i^{-1}_{Y_v,X_v}(\pa i^{-1}_{X_v,X}(\xi)))$.
        It follows that $$|A|\le |\pa i^{-1}_{Y_v,X_v}(\pa i^{-1}_{X_v,X}(\xi))|\le N'\cdot N_{\ref{thm-finite to one TMS ray}}.$$
    \end{proof}

As an application of Theorem \ref{thm-CT subtree of subspaces}, we obtain the natural and obvious analog of Corollaries~\ref{thm-gog vertex} and
\ref{cor-gog-subtree} in the present setup.

The relatively hyperbolic version of Theorem~\ref{thm-main} is still open. We conclude with the following general question.

\begin{question}
    For a hyperbolic group $G$ acting on a proper hyperbolic metric space $X$ without parabolics, assume that the Cannon--Thurston map $\partial i_{G,X}
    : \partial G \to \partial X$ exists.
    Is $\partial i_{G,X}$ always finite-to-one?

    More generally, let $G$ be a relatively hyperbolic group acting on $X$ such that the action is type-preserving. Suppose that the Cannon--Thurston map $\partial i_{G,X}$ exists from the Bowditch boundary of $G$ to $\partial X$.
    Is $\partial i_{G,X}$  finite-to-one away from parabolics?
\end{question}

\newcommand{\etalchar}[1]{$^{#1}$}
\providecommand{\bysame}{\leavevmode\hbox to3em{\hrulefill}\thinspace}
\providecommand{\MR}{\relax\ifhmode\unskip\space\fi MR }
\providecommand{\MRhref}[2]{%
	\href{http://www.ams.org/mathscinet-getitem?mr=#1}{#2}
}
\providecommand{\href}[2]{#2}

\end{document}